\documentclass[amssymb,amscd,12pt,20pt]{amsart}
\usepackage{amsmath,amsthm}
\usepackage{mathrsfs}  
\usepackage{stmaryrd}
\usepackage{amsfonts}
\usepackage{amscd}
\usepackage{tikz}
\usepackage{amssymb}
\usepackage{amscd}
\usepackage[all]{xy}

 \newtheorem{thm}{Theorem}[section]

 \newtheorem{prop}[thm]{Proposition} 
  \theoremstyle{definition}
 \newtheorem{defn}[thm] {Definition} 
 \theoremstyle{remark}
 \newtheorem{rem}[thm]{Remark}

\numberwithin{equation}{subsection}
\newcommand{\OO}{{\mathcal O}}
\newcommand{\I}{{\mathcal I}}
\newcommand{\Q}{{\mathcal Q}}
\newcommand{\LL}{{\mathcal L}}

\newcommand{\ZZ}{{\mathbb Z}}

\newcommand{\RR}{\mathbb R}

\DeclareMathOperator{\Hom}{{Hom}}

\newcommand{\oplusa}[2][]{\underset{#2}{\overset{#1}{\oplus}}}

\newcommand{\os}{\overset}
\newcommand{\wt}{\widetilde}
\newcommand{\ita}{\textit}

\newcommand{\enumera}{\begin{enumerate}}
\newcommand{\eenumera}{\end{enumerate}}

\newcommand{\alineas}[1]{\begin{array}{#1}}
\newcommand{\alinea}{\begin{array}{l}}
\newcommand{\ealinea}{\end{array}}
\newcommand{\ealineas}{\end{array}}

\newcommand{\Qcoh}{\operatorname{Qcoh}_k}
\newcommand{\Qc}{\operatorname{Qc}}

\newcommand{\Shv}{\operatorname{Shv}_k}
\newcommand{\GShv}{\operatorname{{\it G}-Shv}_k}
\newcommand{\Const}{\operatorname{Const}_k}
\newcommand{\LConst}{\operatorname{LConst}_k}

\newcommand{\Cov}{\operatorname{Cov}_k}

\newcommand{\Mod}{\operatorname{Mod}_k}
\newcommand{\Et}{\operatorname{Et}_k}
\newcommand{\Rep}{\operatorname{Rep}_k}

\DeclareMathOperator{\Aut}{{Aut}}

\DeclareMathOperator{\id}{{id}}

\newenvironment{dem}{\noindent{\underline{Proof}:}}{\hfill$\square$}

\newsavebox\CBox

\begin{document}

\title{Asphericity and B\"{o}kstedt-Neeman theorem}

\author{ F. Sancho de Salas}
\author{J.F. Torres Sancho}

\address{ \newline Fernando Sancho de Salas\newline Departamento de
Matem\'aticas and Instituto Universitario de F\'isica Fundamental y Matem\'aticas (IUFFyM)\newline
Universidad de Salamanca\newline  Plaza de la Merced 1-4\\
37008 Salamanca\newline  Spain}
\email{fsancho@usal.es}

\address{ \newline Juan Francisco Torres Sancho\newline Departamento de
Matem\'aticas\newline
Universidad de Salamanca\newline  Plaza de la Merced 1-4\\
37008 Salamanca\newline  Spain}
\email{juanfran24@usal.es}



\thanks {The  authors were supported by research project MTM2017-86042-P (MEC)}

\begin{abstract} We prove that a topological space is aspherical if and only if it satisfies  B\"{o}kstedt-Neeman Theorem, i.e., the derived category of complexes of locally constant sheaves is equivalent to the derived category of complexes of sheaves with locally constant cohomology.   

\end{abstract}

 

\maketitle

\section*{Introduction}  

Let $(X,\OO)$ be a ringed  space and let us denote $\operatorname{Qcoh}(X)$ the category of quasi-coherent $\OO$-modules, which is a faithful subcategory of the category $\operatorname{Mod}(X)$ of all $\OO$-modules. Let us assume that $\operatorname{Qcoh}(X)$ is an abelian subcategory of $\operatorname{Mod}(X)$ and let us denote $D^+(\operatorname{Qcoh}(X)) $ (resp. $D^+(X)$) the derived category of bounded below complexes of quasi-coherent modules (resp. of all $\OO$-modules); finally, let us denote $D_{\text{qc}}^+(X)$ the faithful subcategory of $D^+(X)$ of complexes with quasi-coherent cohomology. One has a natural functor
\[i\colon D^+(\operatorname{Qcoh}(X))\to D_{\text{qc}}^+(X).\]

If $(X,\OO)$ is a quasi-compact and  semi-separated scheme, then $i$ is an equivalence. This is B\"{o}kstedt-Neeman  theorem (see \cite{BokstedtNeeman}). If $(X,\OO)$ is a semi-separated schematic finite space, then B\"{o}kstedt-Neeman  theorem also holds (see \cite{SanchoTorres}). In this paper we investigate the validity of B\"{o}kstedt-Neeman  theorem in the topological case, i.e., when $\OO=\ZZ$ is the constant sheaf (notice that the category of topological spaces is a full  subcategory of that of all ringed spaces by considering any topological space as a ringed space with the constant sheaf $\ZZ$). Our main result is to prove that $(X,\ZZ)$ satisfies B\"{o}kstedt-Neeman theorem if and only if $X$ is aspherical. More generally, for any ring $k$, we shall prove (Theorem \ref{theorem}) that $(X,k)$ satisfies B\"{o}kstedt-Neeman  theorem if and only if $X$ is $k$-aspherical (see Definition \ref{aspherical}).

In section 1 we shall prove that the category of quasi-coherent sheaves coincides with that of locally constant sheaves, on locally connected spaces. We shall also recall the equivalence between locally constant sheaves and representations of the fundamental group. We shall introduce the quasicoherator functor which will play an essential role (as it does on schemes).

Section 2 is devoted to the main result. Also a secondary result is given, about a simple proof of the invariance of the cohomology of locally constant sheaves under a homotopy equivalence (see \cite{KwasikSung}).

\section{Quasi-coherent modules, locally constant sheaves and $G$-modules}

Let $X$ be a topological space and $k$ be a ring. We shall denote by $\Shv(X)$ the category of sheaves of $k$-modules on $X$ and by $\operatorname{Mod}(k)$ the category of $k$-modules. A sheaf always means a sheaf of $k$-modules and a morphim of sheaves means a morphim of sheaves of $k$-modules.

Each $k$-module $M$ defines a constant sheaf $ M_X$ on $X$. We shall denote by $\Const(X)$ the subcategory of $\Shv (X)$ consisting in constant sheaves. If $\pi\colon X\to\{*\}$ is the projection onto a point, then $\pi^{-1}M= M_X$.
\begin{prop}\label{const ab category} If $X$ is connected, the functor $\pi^{-1} \colon \text{\rm Mod} (k)  \to \Const(X)$ is an equivalence. In particular, $\Const(X)$ is an abelian subcategory of $\Shv(X)$. \end{prop}

\begin{proof} For any $M\in \operatorname{Mod}(k)$, the natural morphism $M\to\pi_*\pi^{-1}M$ is an isomorphism because $X$ is connected.
\end{proof}

We shall denote by $\LConst (X)$ the subcategory of $\Shv(X)$ whose objects are the locally constant sheaves on $X$. Locally constant sheaves are also called local systems in the literature. If $X$ is locally connected, then $\LConst(X)$ is an abelian subcategory of $\Shv(X)$ by Proposition \ref{const ab category}.

\begin{defn} A sheaf of $k$-modules $F$ on $X$ is \ita{quasi-coherent} if it is locally a cokernel of free $k$-modules, i.e., for  each point $x \in X$ there exists an open neighbourhood $U_x$ and an exact sequence of sheaves on $U_x$ 
$$  k_{U_x}^{\oplus I} \to  k_{U_x}^{\oplus J} \to F _{\vert U_x}\to 0, $$
where $I$ and $J$ are arbitrary sets of indexes.
\end{defn}

We shall denote by $\Qcoh(X)$ the full subcategory of $\Shv(X)$ whose objects are the quasi-coherent sheaves.

\begin{prop}\label{qc=lc} Any locally constant sheaf is quasi-coherent. Moreover, if $X$ is locally connected, the converse also holds and then $\Qcoh(X)=\LConst(X)$ is an abelian subcategory of $\Shv(X)$.
\end{prop}
\begin{dem} Let $\LL$ be a locally constant sheaf. Since being quasi-coherent is a local question we can assume that $\LL$ is a constant sheaf,  $\LL=  {M_X}$. Now,  $M$ is a cokernel of free $k$-modules, i.e. there is an exact sequence
$$ k^{\oplus I} \to k^{\oplus J} \to M \to 0$$
Aplying the functor $\pi^{-1}$ to this exact sequence, we have the exact sequence
$$  {k_X}^{\oplus I} \to  {k_X}^{\oplus J} \to  {M_X} \to 0$$
hence $ {M_X}$ is quasi-coherent.

Now assume that $\LL$ is quasi-coherent and $X$ is locally connected. Since being locally constant is a local question, we can assume that $X$ is connected and that there exists an exact sequence of sheaves  $ {k_X}^{\oplus I} \to  {k_X}^{\oplus J} \to \LL \to 0$. By Proposition \ref{const ab category}, $\LL$ is constant. 
\end{dem}

\medskip

Let us recall now the equivalence between locally constant sheaves on $X$, coverings of $X$, and $G$-modules (where $G$ is the fundamental group of $X$). Though these results are well known,  they are scattered in the literature and it is difficult to find precise references. We give a brief account of them for the reader's convenience.

\begin{defn} A \ita{$k$-module space over $X$} is a topological space over $X$,  $p \colon X' \to X$, such that

(a) For each $x \in X$, the fibre $p^{-1}(x)$ has a $k$-module structure.

(b) The maps 
$$   X' \times_X X' \overset +\longrightarrow X', \ \ \ \ \ \ \     k \times X' \overset\cdot\to X' $$ 
induced by the $k$-module structure on the fibres are continuous (for the second, $k$ is given the discrete topology). 
\end{defn}

Let $p \colon X' \to X$, $q \colon X'' \to X$ be two $k$-module spaces over $X$. A morphism between them is a continuous map $X' \to X''$ over $X$ compatible with the two operations (i.e, it is fibrewise a morphism of $k$-modules). Let us denote $\operatorname{Top}_{k-\text{mod}}(X)$ the category of $k$-module spaces over $X$.
We shall denote by $\Et(X)$ the faithful subcategory of $\operatorname{Top}_{k-\text{mod}}(X)$ whose objects are the $k$-module spaces over $X$ that are étale (i.e., $X'\to X$ is a local homeomorphism) and by $\Cov(X)$ the full subcategory of $\Et(X)$ whose objects are coverings of $X$.

The well known correspondence between sheaves on $X$ and étale spaces over $X$ yields an equivalence
\[ \Shv(X)= \Et(X)\] that induces an equivalence
\[ \LConst(X)=\Cov(X)\] and  an equivalence between constant sheaves and trivial coverings. These correspondences are functorial: if $f\colon X\to Y$ is a continuous map and $F$ is a sheaf on $Y$, the étalé space associated to $f^{-1}F$ is the pull-back to $X$ of the étalé space associated to $F$.

\subsection{Locally simply connected spaces}

Let us see some additional properties of the category of sheaves on a locally simply connected topological space. By a simply connected space we mean a connected space such that every covering is trivial (this is the case for example if $X$ is path connected and has trivial fundamental group). On a simply connected space $X$ the functor
\[\aligned \Mod(k)&\to \Qcoh(X)\\ M &\mapsto M_X\endaligned\] is an equivalence, whose inverse is the global sections functor $\LL\mapsto \Gamma(X,\LL)$.

\begin{prop}\label{losimcon} Let $X$ be a locally simply connected topological space. If $\{\LL_i\}_{i \in I}$ is a family of quasi-coherent modules  on $X$, then  $\oplusa {i \in I} \LL_i$ is also quasi-coherent. 
\end{prop}
\begin{dem} Let $x \in X$ and let $U_x$ be  a simply connected open neighbourhood  of $x$. Then, $\LL_i |_{U_x}$ is a constant sheaf for all $i \in I$, because any locally constant sheaf on a simply connected space  is constant. Since the direct sum  of constant sheaves is a constant sheaf, we conclude that $\oplus_i\LL_i$ is quasi-coherent.  
\end{dem}

\begin{defn} Let $G$ be a group. A \ita{$G$-$k$-module} (or a $G$-module, for short) is a $k$-module $E$ endowed with an action of $G$ by $k$-linear automorphisms. They are also called $k[G]$-modules or   $k$-linear representations of $G$ in the literature. A morphism of $G$-modules $f \colon E \to E'$  is a morphism of $k$-modules compatible with the action of $G$. We shall denote by $\Mod(G)$  the category of $G$-modules.  Any $k$-module may be thought of as a $G$-module, with the trivial action. These are called {\it trivial} $G$-modules. Thus we have a functor
\[j\colon \operatorname{Mod}(k) \to \Mod(G).\]
\end{defn}

Let $X$ be a connected  and locally simply connected  topological space. Then $X$ admits a universal covering, $\phi \colon \widetilde{X} \to X$. Let $G_X = \Aut_X \widetilde{X}$ (which is isomorphic to the fundamental group of $X$ at any point). Then:

\begin{thm}\label{Qcoh=Rep} One has an equivalence
\[ \Qcoh(X)=\Mod(G_X).\]
\end{thm}

Though this is a well known result (since $\Qcoh(X)=\LConst(X)$), it will be convenient to recall the explicit correspondence. Let us denote $G=G_X$.
 Let us first define a functor
$$\aligned \Mod(G) &\to \Qcoh(X)\\ E&\mapsto \LL_E\endaligned$$ in the following way. Let $E$ be a $G$-module. We endowe the set $E$ with the discrete topology and consider the trivial covering $\widetilde{X} \times E \to \widetilde{X}$. The group $G$ acts on $\widetilde{X} \times E$, $g(\tilde{x},e)= (g\tilde{x},ge)$, and the projection $\widetilde{X} \times E \to \widetilde{X}$ is equivariant. So taking quotient by the action of $G$, we have a continuous map
$$ \wt X_E:=(\widetilde{X} \times E)/G \longrightarrow \widetilde{X}/G = X$$
which is a $k$-module space over $X$ and a covering. Then, $\LL_{E}:=$ sheaf of sections of $\wt X_E \to X$,  is a locally constant sheaf.

Conversely, let us define a functor $$\aligned \Qcoh(X) &\to \Mod(G)\\ \LL&\mapsto E_\LL\endaligned$$ as follows. Let $\LL$ be a quasi-coherent module on $X$. Then $\phi^{-1} \LL$ is a quasi-coherent module $\widetilde{X}$. Since $\widetilde{X}$ is simply connected, $\phi^{-1} \LL$  is a constant sheaf. Thus $\phi^{-1} \LL=(E_{\LL})_{\widetilde{X}}$ with $E_{\LL}:= \Gamma(\wt X, \phi^{-1} \LL)$. The group $G$ acts in $\phi^{-1} \LL$, so it acts in $E_{\LL}$,  i.e. $E_\LL$ is a $G$-module. 

The assignations $E\mapsto \LL_E$ and $\LL\to E_\LL$ give the equivalence of Theorem \ref{Qcoh=Rep}.


\begin{rem}\label{obs}  
The equivalence of  Theorem \ref{Qcoh=Rep} restricts to an equivalence between the category of  constant sheaves and the category of trivial $G$-modules, giving us a conmutative diagram
$$\xymatrix{   \text{Mod}(k) \ar[r]^{  \pi^{-1}} \ar@2{-}[d]  & \Qcoh (X)  \ar@2{-}^{\text{Th.}\ref{Qcoh=Rep}}[d]  \\
 \text{Mod} (k) \ar[r]^{j\ \ }   & \Rep(X)  }$$
where $\pi \colon X \to \{*\}$ is the projection to a point. Since the right adjoint  of $\pi^{-1}$ is $\pi_*=\Gamma(X,-)$ and the right ajoint of $j$ is the ``taking invariants'' functor $(-)^G$, we have a conmutative diagram
$$\xymatrix{ \Qcoh (X) \ar[r]^{\Gamma(X,-)} \ar@2{-}[d]_{\text{Th.}\ref{Qcoh=Rep}}  & \text{Mod}(k) \ar@2{-}[d] \\
\Rep(X) \ar[r]^{(-)^G} & \text{Mod} (k)}$$ that is, $$E^G=\Gamma(X,\LL_E)$$ for any $G$-module $E$.
\end{rem}		
		
%
%

\subsection{Quasicoherator functor}

In all this section, we assume $X$ is connected and locally simply connected, $\phi\colon \wt X\to X$ is a universal covering and $G=\Aut_X\wt X$ is the fundamental group.

\begin{defn} A quasi-coherent sheaf $\I$ is called {\em qc-injective} if it is an injective object of the category $\Qcoh(X)$, i.e., if $\Hom(\quad,\I)$ is exact on quasi-coherent sheaves. Notice that a qc-injective quasi-coherent sheaf is {\it not} an injective sheaf in general. By the equivalence of Theorem \ref{Qcoh=Rep}, $\I$ is qc-injective if and only if its corresponding $G$-module $E_\I$ is an injective $G$-module. 
\end{defn}

 Let us consider the inclusion functor
$$ i \colon \Qcoh (X) \hookrightarrow \text{Shv}_k (X).$$

\begin{defn} Since $i$ is an exact functor and commutes with direct sums (Proposition \ref{losimcon}), it has a right adjoint (by Grothendieck's representability theorem)
$$\Qc \colon  \Shv  (X) \to \Qcoh (X).$$
Thus,  for any $\LL \in \Qcoh (X)$ and any $F \in \Shv  (X)$ there is a natural isomorphism
$$ \Hom(i(\LL),F) = \Hom(\LL,\Qc(F)).$$ 
\end{defn}
 
$\Qc$ is left exact and it is the identity on $\Qcoh(X)$: for any quasi-coherent sheaf $\LL$ one $X$, one has $\Qc(\LL)=\LL$. Since $i$ is exact, $\Qc$ takes injective sheaves into qc-injective quasi-coherent sheaves.

Let us give a more explicit description of the functor $\Qc$ in terms of the universal covering $\phi\colon \wt X\to X$. The functor   $\Qcoh(X)  \os{\sim}{\to}  \Mod(G_X), \LL  \longmapsto  E_{\LL}= \Gamma(\widetilde{X}, \phi^{-1} \LL)$ can be extended to the category of all  sheaves of $k$-modules; that is, we have a functor $\Shv(X) \to \Mod(G_X)$, $F  \mapsto   \Gamma(\widetilde{X},\phi^{-1} F)$. Let us see that this functor coincides with $\Qc$ (via the equivalence $\Qcoh(X)=\Mod (G_X)$).

\begin{prop}\label{Qc=Gamma} For any sheaf $F$ on $X$, $\Qc(F)$ is the quasi-coherent sheaf corresponding to the $G$-module $\Gamma(\wt X,\phi^{-1}F)$. Thus, we shall put
\[ \Qc(F)=\Gamma (\wt X,\phi^{-1}F)\] via the equivalence of Theorem \ref{Qcoh=Rep}.
\end{prop}

\begin{proof} This is essentially a consequence of the Galois correspondence between sheaves on $X$ and $G$-sheaves on $\wt X$, that we recall now.  A $G$-sheaf on $\wt X$ is a sheaf $\Q$ on $\wt X$ endowed with an action of $G$ which is compatible with the action of $G$ on $\wt X$; this means that if $\wt\Q\to \wt X$ denotes the étalé space associated to $\Q$, then $G$ acts on $\wt Q$ and the map $\wt\Q\to \wt X$ is equivariant. Morphisms of $G$-sheaves are morphisms of sheaves compatible with the action of $G$. We denote by $\GShv(\wt X)$ the category of $G$-sheaves on $\wt X$.

If $F$ is a sheaf on $X$, then $\phi^{-1}F$ is a $G$-sheaf on $\wt X$, since $\wt{\phi^{-1}F}=\wt X\times_X\wt F$ is endowed with the action of $G$ given by the action on $\wt X$ and the trivial action on $\wt F$. One has then a Galois equivalence
\[ \aligned \Shv(X)&\to  \GShv(\wt X)\\ F&\mapsto \phi^{-1}F\endaligned\] whose inverse sends a $G$-sheaf $\Q$ on $\wt X$ to  the sheaf of sections of $\wt F:=\wt Q/G\to \wt X/G=X$.

If $E$ is a $G$-module, then $E_{\wt X}$ is a (constant) $G$-sheaf on $\wt X$ and for any $G$-sheaf $\Q$ on $\wt X$ $\Gamma(\wt X,\Q)$ is a $G$-module. The functor $E\mapsto E_{\wt X}$ is left adjoint of the functor $\Q\mapsto \Gamma(\wt X,\Q)$.

Now, let us prove the proposition. Let us denote $E_F=\Gamma(\wt X,\phi^{-1}F)$. For any quasi-coherent sheaf $\LL$ on $X$ one has
\[\Hom_{\Qcoh(X)}(\LL,\Qc(F))=\Hom_{\Shv(X)}(\LL,F)=\Hom_{\GShv(\wt X)}(\phi^{-1}\LL,\phi^{-1}F).\]
Now, $\phi^{-1}\LL$ is a constant $G$-sheaf on $\wt X$, hence $\phi^{-1}\LL= (E_\LL)_{\wt X}$. 
Then
\[ \Hom_{\GShv(\wt X)}(\phi^{-1}\LL,\phi^{-1}F)= \Hom_{\text{$G$-mod}}(E_\LL,E_F)\overset{\text{Th.}\ref{Qcoh=Rep}}=\Hom_{\Qcoh(X)}(\LL,\LL_{E_F})\] and we conclude that $\Qc(F)=\LL_{E_F}$.
\end{proof}

\section{Asphericity and B\"{o}kstedt-Neeman Theorem} 

  We shall denote by $D^+(X):=D^+(\Shv(X))$ the derived category of bounded below complexes of sheaves of $k$-modules on $X$, $D^+(\Qcoh(X))$ the derived category of bounded below complexes of quasi-coherent sheaves of $k$-modules on $X$ and by $D_{\text{qc}}^+(X)$ the full subcategory of $D^+(X)$ whose objects are the complexes of sheaves with quasi-coherent cohomology. We still denote by $i$ the natural functor
  $$ i \colon D^+(\Qcoh (X)) \to D^+(X)$$ which takes values in $D_{\text{qc}}^+(X)$. Notice that though $i \colon \Qcoh(X) \hookrightarrow \Shv(X)$ is fully faithful, $ i \colon  D^+(\Qcoh (X)) \to  D^+(X)$ is not in general (as we shall see,  $i$ is fully faithful precisely when $X$ is aspherical). 
 
Let $$\RR \Qc \colon  D^+(X) \to  D^+(\Qcoh (X))$$ be the right derived functor of $\Qc\colon \Shv(X)\to \Qcoh(X)$. It is  right adjoint of $i$. Hence we have natural transformations
\[ \epsilon\colon \RR\Qc\circ i\to \id,\qquad \mu\colon i\circ\RR\Qc\to \id.\] 

Let $$\RR \Gamma(X,-)\colon D^+(X) \to D^+(k)$$ be the right derived functor of the global sections functor $\Gamma(X,-) \colon \Shv(X) \to \text{Mod}(k)$. For any $F\in D^+(X)$, we shall denote $$H^i(X,F)=H^i[\RR \Gamma(X,F)]$$ the i-th cohomology module of $F$.

Let us denote $$\RR_{\text{qc}} \Gamma(X,-)\colon D^+(\Qcoh (X)) \to D^+(k)$$ the right derived functor of $ \Gamma(X,-) \colon \Qcoh(X) \to \text{Mod}(k)$. For any  $\LL\in D^+(\Qcoh(X))$, $$H^i_{\text{qc}}(X,\LL):=H^i[\RR_{\text{qc}}\Gamma(X,\LL)]$$ shall be called the i-th {\em quasi-coherent cohomology module of $\LL$}.

We shall usually use the abbreviated notations $\RR\Gamma:=\RR\Gamma(X,\quad) $ and $\RR_{\text{qc}} \Gamma:=\RR_{\text{qc}} \Gamma(X,\quad)$.

Note that $\RR \Gamma$ is defined via resolutions by injective sheaves while $\RR_{\text{qc}} \Gamma$ is defined via resolutions by qc-injective quasi-coherent sheaves. In consequence, the triangle
$$\xymatrix{  D^+(\Qcoh (X))  \ar[rd]_{\RR_{\text{qc}}\Gamma} \ar[rr]^{i} &   & D^+X) \ar[ld]^{\RR \Gamma} \\
&  D^+(k) &    }$$
is not conmutative in general, because qc-injective quasi-coherent  sheaves are not $\Gamma(X,\quad)$-acyclic in general. However, there is a natural morphism 
$$\RR_{\text{qc}} \Gamma \to \RR \Gamma \circ i $$
which is not an isomorphism in general. In consequence, for any  $\LL\in D^+(\Qcoh(X))$  there is a natural morphism
$$ H^i_{\text{qc}}(X, \LL) \to H^i(X, \LL)$$
which is not an isomorphism in general. 

\begin{rem} By Remark \ref{obs},   {\em quasi-coherent cohomology coincides with  cohomology of groups}: For any $E\in D^+(\Mod(G))$
\[ H^i(G,E)=H^i_{\text{qc}}(X,\LL_E)\] where $H^i(G,E)$ stands for cohomology of groups and $\LL_E\in D^+(\Qcoh(X))$ is the object corresponding to $E$ by the equivalence $D^+(\Qcoh(X))=D^+(\Mod(G))$ induced by $\Qcoh(X)=\Mod(G)$.
\end{rem}

On the other hand, one has:

\begin{prop} \label{rqc isomorhism} The diagram
$$\xymatrix{  D^+(X)  \ar[rd]_{\RR\Gamma} \ar[rr]^{\RR \Qc} &   & D^+(\Qcoh (X)) \ar[ld]^{\RR_{\text{\rm qc}} \Gamma} \\
&  D^+(k) &    }$$
is conmutative; that is, there is a natural isomorphism 
$$ \RR \Gamma \simeq \RR_{\text{\rm qc}}\Gamma \circ \RR \Qc$$ and then a convergent spectral sequence
\[ H^p_{\text{\rm qc}}(X,R^q\Qc(F))\Rightarrow H^{p+q}(X,F).\]
\end{prop}
\begin{dem}  This is Grothendieck composite functor theorem (and Grothendieck spectral sequence), since $\Gamma(X,F)=\Gamma(X,\Qc(F))$ and $\Qc$ takes injective sheaves into qc-injectives quasi-coherent sheaves.
\end{dem}
\medskip

A description of  $R^i \Qc$ in terms of the universal covering is given by the following:
\begin{prop} \label{rqc description} For any $F \in D^+(X)$, one has that
$$ R^i \Qc(F) =H^i (\wt{X}, \phi^{-1} F)$$ via the equivalence $\Qcoh(X)=\Mod(G)$.
\end{prop}
\begin{dem} Let $F\to I$ be an injective resolution. Then  $$R^i\Qc (F)=H^i(\Qc(I))=H^i\Gamma(\wt X,\phi^{-1}I)$$ by Proposition \ref{Qc=Gamma}. Since $\phi^{-1}I$ is a resolution of $\phi^{-1}F$, we conclude if we prove that $\phi^{-1}$ takes injective sheaves into injective sheaves. More generally, let us see that for any covering $\phi\colon X'\to X$ the functor $\phi^{-1}\colon \Shv(X)\to \Shv (  X')$ takes injective sheaves into injective sheaves. Since $\phi^{-1}$  is exact and commutes with direct products (because $\phi$ is a local homeomorphism), it has a left adjoint $\phi_!\colon \Shv(X')\to \Shv(X)$ and we conclude if we prove that $\phi_!$ is exact. Since $\phi$ is a local homeomorphism, for any open subset $j\colon U\hookrightarrow X$ and any sheaf $F'$ on $U$ one has an isomorphism
\[ \phi^{-1}(j_*F')=j'_*\phi_U^{-1}F'\] where $\phi_U\colon  \phi^{-1}(U)\to U$ is the restriction of $\phi$ to $ \phi^{-1}(U)$ and $j'$ is the inclusion of $\phi^{-1}(U)$ in $X'$. By adjunction, for any sheaf
$\Q$ on $X'$, one has
\[  (\phi_! F)_{\vert U}= (\phi_U)_!(F_{\vert \phi^{-1}(U)}).\]  Hence, the exactness of $\phi_!$ is reduced to the case that $\phi\colon X'\to X$ is a trivial covering: $X'=X\times I$ for some discrete space $I$. But then a sheaf $F$ on $X'$ is equivalent to a family of sheaves $\{ F_i\}_{i\in I}$ on $X$ and it is easy  to see that $\phi_!F=\underset{i\in I}\oplus F_i$, which is clearly an exact functor.
\end{dem}

Propositions \ref{rqc isomorhism} and  \ref{rqc description} may be used to give an elementary proof that cohomology of locally constant sheaves is invariant under homotopy (see \cite{KwasikSung}), as we shall see now.

Let $f \colon X \to Y$ be a continuous map (between connected and locally simply connected spaces). Let $F\in D^+(Y)$   and let us consider the induced morphism in cohomology
\[ H^i(Y , \LL) \to H^i(X, f^{-1} \LL).\]

\begin{prop} If $f$ is a homotopy equivalence and $F\in D^+_{\text{\rm qc}}(Y)$, then the morphism 
\[ H^i(Y , \LL) \to H^i(X, f^{-1} \LL) \] is an isomorphism.
\end{prop}
\begin{dem} Let $\phi_X\colon \wt X\to X$, $\phi_Y\colon \wt Y\to Y$ be universal coverings and $\wt f\colon \wt X\to\wt Y$ the induced continuous map (fixing base points). Then $\wt f$ is also an homotopy equivalence and induces an isomorphism of groups $\wt f\colon G_X\to G_Y$. Let us consider the convergent spectral sequences
\[ H^p_{\text{\rm qc}}(Y,R^q\Qc(F))\Rightarrow H^{p+q}(Y,F),\qquad   H^p_{\text{\rm qc}}(X,R^q\Qc(\phi^{-1}F))\Rightarrow H^{p+q}(X,\phi^{-1}F).\]

It sufices to see that the natural morphism $H^p_{\text{\rm qc}}(Y,R^q\Qc(F))\to H^p_{\text{\rm qc}}(X,R^q\Qc(\phi^{-1}F))$ is an isomorphism. If we prove that $f^{-1} R^q\Qc(F)\to  R^q\Qc(\phi^{-1}F)$ is an isomorphism, we conclude (because, since $f$ is a homotopy equivalence,  $f^{-1}\colon\Qcoh(Y)\to\Qcoh(X)$ is an equivalence by Theorem \ref{Qcoh=Rep}). Now, $f^{-1} R^q\Qc(F)\to  R^q\Qc(\phi^{-1}F)$ is an isomorphism by Proposition \ref{rqc description}, since $H^q(\wt Y, \phi_Y^{-1}F)\to H^q(\wt X, {\wt f}^{-1}\phi_Y^{-1}F)= H^q(\wt X, \phi_X^{-1}f^{-1}F)$ is an isomorphism because $\phi_Y^{-1}F$ is a complex with {\em constant} cohomology and $\wt f$ is an homotopy equivalence (we are using the invariance of the cohomology of constant sheaves under homotopy).
\end{dem} 

\begin{rem} Proposition \ref{rqc isomorhism} is not new, since it may be interpretated, via Theorem \ref{Qcoh=Rep}, in terms of cohomology of groups in the following way: for any $F\in D^+(X)$, one has an isomorphism
\[ \RR\Gamma(X,F)\simeq \RR\Gamma(G,\RR(\Gamma \phi^{-1})(F))\] where $\RR\Gamma(G,\quad)$ is the right derived functor of the functor of invariants $(\underline\quad)^G\colon\Mod(G)\to \operatorname{Mod}(k)$, and $\RR(\Gamma \phi^{-1})\colon D^+(X)\to D^+(\Mod(G))$ is the right derived functor of the functor $\Gamma\phi^{-1}\colon \Shv(X)\to \Mod(G)$, $F\mapsto \Gamma(\wt X,\phi^{-1}F)$. The associated spectral sequence is
\[ H^p(G,H^q(\wt X,\phi^{-1}F))\Rightarrow H^{p+q}(X,F).\]
\end{rem}

\subsection{Aspherical spaces} $\,$

Let $X$ be a connected and locally simply connected topological space and $\phi \colon \widetilde{X} \to X$ a universal covering.

\begin{defn}\label{aspherical} We say that $X$ is a  \textit{$k$-aspherical space} if $H^i(\widetilde{X}, M_{\wt X})=0$, for all $i >0$ and for any constant sheaf of $k$-modules $M_{\wt X}$ on $\widetilde{X}$.
\end{defn}

\begin{rem} If $k=\ZZ$, then universal coefficient formula for cohomology says that $X$ is $\ZZ$-aspherical if and only if the homology groups $H_i(\wt X,\ZZ)$  vanish for all $i>0$. By Hurewicz theorem, this is equivalent to say that $\pi_n(\wt X)=0$ for $n\geq 2$,  i.e., $\pi_n(X)=0$ for $n\geq 2$, which is the usual definition of an {\it aspherical space}.
\end{rem}

\begin{prop}\label{aspherical-RQc} $X$ is $k$-aspherical if and only if
\[ R^i\Qc(\LL)=0\] for any $i>0$ and any quasi-coherent sheaf $\LL$.
\end{prop}

\begin{proof} This is a consequence of Proposition \ref{rqc description}. Indeed, assume that $X$ is $k$-aspherical. By Proposition \ref{rqc description}
\[ R^i\Qc(\LL)=H^i(\wt X,\phi^{-1}\LL)=0\] for $i>0$ because $\phi^{-1}\LL$ is a constant sheaf on $\wt X$ and $\wt X$ is aspherical. Conversely, if  $R^i\Qc(\LL)=0$ for any $i>0$ and any quasi-coherent sheaf $\LL$, then, for any $k$-module $M$ and any $i>0$
\[ H^i(\wt X, M_{\wt X})= H^i(\wt X,\phi^{-1} M_{X})\overset{\text{Prop.}\ref{rqc description}}= R^i\Qc(M_X)=0.\]
\end{proof}

\begin{prop}\label{BN-char} The following conditions are equivalent.
\begin{enumerate} \item $ i \colon D^+(\Qcoh (X)) \to  D^+_{\text{\rm qc}}(X)$  is an equivalence.
\item $i$ is fully faithful; equivalently, $\epsilon\colon \id\to \RR\Qc\circ i $ is an isomorphism.
\item $R^i\Qc(\LL)=0$ for any $i>0$ and any quasi-coherent sheaf $\LL$.
\item $ \RR\Qc \colon D^+_{\text{\rm qc}}(X) \to D^+(\Qcoh (X))   $  is fully faithful; equivalently, $\mu\colon i\circ \RR\Qc \to  \id$ is an isomorphism.
\item $ \RR\Qc \colon D^+_{\text{\rm qc}}(X) \to D^+(\Qcoh (X))   $  is an equivalence.
\end{enumerate}
\end{prop}

\begin{proof} $1)\Rightarrow 2)\Rightarrow 3)$ are immediate. Let us see that $3)\Rightarrow 4)$. Let $F\in D^+_{\text{qc}}(X)$. By hypothesis, $R^p\Qc (H^q(F))=0$ for $p>0$ and any $q$. Hence, $R^n\Qc(F)=\Qc(H^n(F))=H^n(F)$, and then $i(\RR\Qc(F))\to F$ is an isomorphism.

$4)\Rightarrow 5)$ It suffices to see that $\epsilon(\LL)\colon \LL\to \RR\Qc( i(\LL)) $ is an isomorphism; for this, it suffices to see the isomorphism after applying $i$. Now, the composition $i(\LL)\overset{i(\epsilon(\LL))}\to i( \RR\Qc ( i(\LL))) \overset{\mu(i(\LL))}\to i(\LL)$ is the identity and $\mu(i(\LL))$ is an isomorphism by hypothesis, hence $\epsilon(i(\LL))$ is an isomorphism.

$5)\Rightarrow 1)$ Since $i$ is left adjoint of $\RR\Qc$, this is immediate.
\end{proof}

\begin{thm}\label{theorem} Let $X$ be a connected and locally simply connected space. The following conditions are equivalent:
\begin{enumerate}
\item $X$ satisfies B\"{o}kstedt-Neeman theorem: The functor $ i \colon D^+(\Qcoh (X)) \to  D^+_{\text{\rm qc}}(X)$  is an equivalence.

\item $X$ is  $k$-aspherical.

\item The triangle
$$\xymatrix{  D^+(\Qcoh (X))  \ar[rd]_{\RR_{\text{\rm qc}}\Gamma} \ar[rr]^{i} &   & D^+_{\text{qc}}(X) \ar[ld]^{\RR \Gamma} \\
&  D^+(k) &    }$$
is conmutative, i.e. the natural morphism $\RR_{\text{\rm qc}}\Gamma\to \RR\Gamma\circ i$ is an isomorphism.

\item Cohomology of groups coincides with sheaf-cohomology: if $G$ is the fundamental group of $X$, then, for any $G$-module $E$
\[ H^i(G,E)=H^i(X,\LL_E)\] where $\LL_E$ is the locally constant sheaf associated to $E$.
\end{enumerate}
\end{thm}

\begin{dem} \noindent $(1) \Rightarrow (3)$. By Proposition \ref{rqc isomorhism}, $\RR_{\text{qc}} \circ \RR\Qc \simeq \RR \Gamma$. Composing with $i$, one has $$\RR_{\text{qc}}\Gamma \circ \RR\Qc\circ i \simeq \RR \Gamma \circ i.$$  We conclude because $\RR\Qc\circ i\simeq \id$ by hypothesis.

$(3)\Rightarrow (4)$. By hypothesis, for any quasi-coherent sheaf $\LL$, the natural morphism $H^i_\text{qc}(X,\LL)\to H^i(X,\LL)$ is an isomorphism. Since $H^i(G,E)=H^i_\text{qc}(X,\LL_E)$, we are done.

$(4)\Rightarrow (2)$. By Proposition \ref{aspherical-RQc}, it suffices to prove that $R^i\Qc(\LL)=0$ for any $i>0$ and any quasi-coherent sheaf $\LL$. 

Let us prove first that any qc-injective sheaf $\I$ is $\Qc$-acyclic. By hypothesis $H^i(X,\I)=H^i(G,E_\I)=0$ for $i>0$ (since $E_\I$ is an injective $G$-module). Now
\[ R^i\Qc(\I)\overset {\ref{rqc description}}=H^i(\wt X,\phi^{-1}\I)=H^i(X,\phi_*\phi^{-1}\I)\] because $R^i\phi_*M_{\wt X}=0$ for any $i>0$ and any constant sheaf $M_{\wt X}$ on $\wt X$. If we prove that $\phi_*\phi^{-1}\I$ is a qc-injective sheaf, we conclude that $R^i\Qc(\I)=0$. Since $\phi_*$ preserves quasi-coherence, we have only to see the qc-injectivity of $\phi_*\phi^{-1}\I$. For any $\LL\in\Qcoh(X)$ one has
$$ \Hom_{\Qcoh(X)}(\LL,\phi_* \phi^{-1} I) = \Hom_{\Qcoh(\wt X)}(\phi^{-1} \LL, \phi^{-1} I) = \Hom_{k-\text{mod}}(E_\LL, E_\I)  $$ and $E_\I$ is an injective $k$-module because it is an injective $G$-module. This proves the qc-injectivity of $\phi_*\phi^{-1}\I$.

Let us conclude now that $R^i\Qc(\LL)=0$ for any $i>0$ and any quasi-coherent sheaf $\LL$. Let $\LL\to\I$ be a qc-injective resolution of $\LL$. Then, by De Rham's theorem,
\[ R^i\Qc(\LL)=H^i[\Qc(\I)]=H^i(\I)=0\] for $i>0$.

$(2) \Leftrightarrow (1)$ It follows from Propositions \ref{aspherical-RQc} and \ref{BN-char}.
\end{dem}

\begin{rem} Let us see the similarity between the B\"{o}kstedt-Neeman theorem on schemes and on topological spaces. Let us recall the notion of affine ringed space given in \cite{Sanchos}. Let $(X,\OO)$ be a ringed space, $A=\OO(X)$ and $\pi\colon (X,\OO)\to (*,A)$ the natural morphism of ringed spaces (where $*$ denotes the topological space with one  point). Then $(X,\OO)$ is called {\em affine} if it satisfies:

(1) The natural functor: $\pi^*\colon \{\text{$A$-modules}\}\to \operatorname{Qcoh}(X)$ is an equivalence.

(2) $H^i(X,\OO)=0$ for $i>0$.

If $X$ is a scheme, then it is affine in the previous sense if and only if it is an affine scheme in the ordinary sense. In the topological case (i.e., $\OO=\ZZ$ is the constant sheaf) then $(X,\ZZ)$ satisfies (1) if and only if $X$ is simply connected and then $(X,\ZZ)$ is affine if and only if $X$ is homotopically trivial (see \cite{Sanchos} for details). Thus, $X$ is aspherical if and only if $\wt X$ is affine.

Let $X$ be a quasi-compact semi-separated scheme,   $U_1,\dots,U_n$   a finite open covering by  affine schemes and let us denote $\wt X=\underset i\coprod U_i$ and $\phi\colon \wt X\to X$ the natural morphism. Notice that $\wt X$ is an affine scheme, $\phi$ is an affine morphism (because $X$ is semi-separated) and $\phi$ is faithfully flat. The validity of B\"{o}kstedt-Neeman theorem relies essentially on two steps: first, it holds for affine schemes (hence for $\wt X$); then it holds for $X$ because $\phi$ is affine and  faithfully flat.

In the topological case, let $\wt X$ be the universal covering of $X$. It still holds that $\phi\colon \wt X\to X$ is faithfully flat ($f^{-1}$ is exact and $f^{-1}F=0$ implies $F=0$) and affine (any connected component of the  preimage of an affine open subset is affine). Thus, the validity of B\"{o}kstedt-Neeman on $X$ is reduced to  the affineness of $\wt X$, which is precisely the asphericity of $X$.
\end{rem}


\begin{thebibliography}{1}	

%
%
%
%
%
%
%
%



\bibitem{BokstedtNeeman} {\sc M.~B\"{o}kstedt and A.~Neeman}, {\em Homotopy limits in triangulated categories}, Compositio Math.  86,  no. 2, 209-234   (1993).
 


\bibitem{KwasikSung} {\sc S.~Kwasik and F.~Sung}, {\em Local coefficients revisited}, Enseign. Math. 64, no. 3-4, 283–303 (2018).

 
 


%
%
%
%
%
%
%

\bibitem{SanchoTorres}  {\sc F.~Sancho de Salas and J.F.~Torres Sancho}, {\em Derived categories of finite spaces and Grothendieck duality}, Mediterr. J. Math. 17, no. 3, Paper No. 80, 22 pp.   (2020).


\bibitem{Sanchos} {\sc F.~Sancho de Salas and P.~Sancho de Salas}, {\em Affine ringed spaces and Serre's criterion}, Rocky Mountain J. Math.  47,  no. 6, 2051-2081   (2017).


 


\end{thebibliography}
\end{document}